\documentclass{amsart}
\usepackage{graphicx}
\usepackage{latexsym}
\usepackage{amsfonts}
\usepackage[all]{xy}
\usepackage{amssymb, mathrsfs, amsfonts, amsmath}
\usepackage{amsbsy}
\usepackage{amsfonts}
\newtheorem*{acknowledgement}{Acknowledgement}
\newtheorem{corollary}{Corollary}
\newtheorem{definition}{Definition}
\newtheorem{lemma}{Lemma}

\newtheorem{theorem}{Theorem}
\newtheorem{example}{Example}
\numberwithin{equation}{section}

\begin{document}
\title[$3$-Homogeneous quasi-Einstein metrics]{Uniqueness of quasi-Einstein metrics\\ on 3-dimensional homogeneous manifolds}
\author{A. Barros $^{1}$, \,E. Ribeiro Jr $^{2}$ \& J. Silva Filho $^{3}$}
\address{$^{1,2}$Universidade Federal do Cear\'a - UFC, Departamento  de Matem\'atica, Campus do Pici, Av. Humberto Monte, Bloco 914,
60455-760-Fortaleza / CE , Brazil}
\email{abbarros@mat.ufc.br}
\email{ernani@mat.ufc.br}

\thanks{$^{1}$ Partially supported by CNPq/Brazil}
\thanks{$^{2}$ Partially supported by grants from  PJP-FUNCAP/Brazil and CNPq/Brazil}

\noindent \address{$^{3}$ Instituto de Ci\^encias exatas e da Natureza - UNILAB, Campus dos Palmares, CE 060, Km 51, 62785-000-Acarape-CE/Brazil}
\email{joaofilho@unilab.edu.br}
\thanks{$^{3}$ Partially supported by CNPq/Brazil}
\keywords{quasi-Einstein metrics, homogeneous manifolds, isometry group, static metrics}
\subjclass[2000]{Primary 53C25, 53C20, 53C21; Secondary 53C65}
%\urladdr{http://www.mat.ufc.br/pgmat}
\date{January 31, 2013}

\begin{abstract}
The purpose of this article is to study the existence and uniqueness of quasi-Einstein structures on $3$-dimensional homogeneous Riemannian manifolds. To this end, we use the eight model geometries for 3-dimensional manifolds identified by Thurston. First,  we
present here a complete description of quasi-Einstein metrics
on  $3$-dimensional homogeneous manifolds with isometry group of dimension $4.$ In addition, we shall show the
absence of such gradient structure on $Sol^3,$ which has $3$-dimensional isometry group.  Moreover, we prove that Berger's spheres carry a non-trivial quasi-Einstein structure with non gradient associated vector field, this shows that a theorem due to Perelman can not be extend to quasi-Einstein metrics. Finally, we prove that a $3$-dimensional homogeneous manifold carrying a gradient quasi-Einstein structure is either Einstein or $\mathbb{H}^2_{\kappa} \times \mathbb{R}.$
\end{abstract}

\maketitle

\section{Introduction}
One of the motivation to study quasi-Einstein metrics on a Riemannian manifold
$(M^n,\,g)$  is its closed relation with  warped product Einstein metrics,
see e.g. \cite{case}, \cite{csw} and \cite{hepeterwylie}. From notable books as  \cite{besse} and \cite{Barrett} classifying or understanding the geometry of Einstein warped products  is definitely a fruitful problem. Indeed, from \cite{kk} under suitable condition, $m$-quasi-Einstein metrics correspond to exactly
those $n$-dimensional manifolds which are the base of an $(n+m)$-dimensional Einstein
warped product.  For comprehensive references on such a theory, we indicate for instance
\cite{besse}, \cite{csw}, \cite{kk} and \cite{WW}.

One fundamental tool to understand the behavior of such
a class of manifold is the $m$-Bakry-Emery Ricci tensor which is given by
\begin{equation}
\label{bertens}
Ric_{f}^{m}=Ric+\nabla ^2f-\frac{1}{m}df\otimes df,
\end{equation}where $f$ is a smooth function on $M^n$ and $\nabla ^2f$ stands for the Hessian form.

This tensor was extended recently, independently, by Barros and Ribeiro Jr \cite{brG} and Limoncu \cite{limoncu} for an arbitrary vector field $X$ on $M^n$ as follows:
\begin{equation}
\label{eqprincqem}
Ric_{X}^{m}=Ric+\frac{1}{2}\mathcal{L}_{X}{g}-\frac{1}{m}X^{\flat}\otimes X^{\flat},
\end{equation}where $\mathcal{L}_{X}{g}$ and $X^{\flat}$ denote, respectively, the Lie derivative on $M^n$ and the canonical $1$-form associated to $X.$

With this setting we say that $(M^n,\,g)$ is a $m$-quasi-Einstein metric, or simply {\it quasi-Einstein metric}, if there exist a vector field $X\in\mathfrak{X}(M)$ and constants $0<m\leq\infty$ and $\lambda$ such that
\begin{equation}
\label{eqprincqem1}
Ric_{X}^{m}=\lambda g.
\end{equation}

Noticing that the trace of $Ric_{X}^{m}$ is given by $R+div X-\frac{1}{m}|X|^{2}$, where $R$ denotes the scalar curvature, we deduce
\begin{equation}
\label{eqntrace}
R+div X-\frac{1}{m}|X|^{2}=\lambda n.
\end{equation}

On the other hand, when $m$ goes to infinity, equation (\ref{eqprincqem}) reduces to the one associated to a Ricci soliton, for more details in this subject we recommend  the survey due to Cao \cite{Cao} and the references therein. Whereas, when $m$ is a positive integer and $X$ is gradient, it corresponds to warped product Einstein metric, for more details see \cite{kk} and \cite{csw}. Following the terminology of Ricci solitons, a  quasi-Einstein metric $g$ on a manifold $M^n$ will be called \emph{expanding}, \emph{steady} or \emph{shrinking}, respectively, if  $\lambda<0,\,\lambda=0$ or $\lambda>0$.

\begin{definition}
A quasi-Einstein metric will be called \emph{trivial} if $ X\equiv0$. Otherwise, it will be \emph{nontrivial}
\end{definition}

We notice that the triviality implies that $M^n$ is an Einstein manifold. Moreover, it is important to detach that gradient  $1$-quasi-Einstein metrics satisfying $\Delta e^{-f}+\lambda e^{-f}=0$ are more commonly called \emph{static metrics} with cosmological constant $\lambda;$ for more details see e.g. \cite{Anderson}, \cite{AndersonKhuri} and \cite{Corvino}. On the other hand, it is well known that on a compact manifold $M^n$ a gradient $\infty-$quasi-Einstein metric with $\lambda\leq0$ is trivial, see \cite{monte}. The same result was proved in \cite{kk} for gradient $m$-quasi-Einstein metric on compact manifold with $m$ finite. Besides, we known that compact shrinking Ricci solitons have positive scalar curvature, see e.g. \cite{monte}. An extension of this result for shrinking gradient $m$-quasi-Einstein metrics with  $1\leq m<\infty$ was obtained in \cite{csw}. Recently, in \cite{brozosGarciagavino} Brozos-V\'{a}zquez et al. proved that locally conformally flat gradient $m$-quasi-Einstein metrics are globally conformally equivalent to a space form or locally isometric to a Robertson-Walker spacetime or a $pp$-wave. In \cite{hepeterwylie} it was given some classification for $m$-quasi-Einstein metrics where the base has non empty boundary. Moreover, they proved a characterization for $m$-quasi-Einstein metrics when the base is locally conformally flat. We point out that Case et al. in \cite{csw} proved that every compact gradient $m$-quasi-Einstein metric with constant scalar curvature is trivial.  We also recall a classical result due to Perelman \cite{Per} which asserts that every compact Ricci soliton is gradient. It is natural to ask if this result still true for quasi-Einstein metrics.

Here we shall show that Berger's spheres carry naturally  a non trivial structure of quasi-Einstein metrics (cf. Example \ref{ex1qem}). Since they have constant scalar curvature, their associated vector fields can not be gradient. In particular, this shows that Perelman's result is not true for compact quasi Einstein metrics. Moreover, these examples show that Theorem 4.6 of \cite{hepeterwylie} can not be extended for a non gradient vector field. Before to do this, we shall recall some classical informations on 3-dimensional homogeneous manifolds.

It should be emphasized that, classically, the study of quasi-Einstein metric is considered when $X$ is a
gradient of a smooth function $f$ on $M^n,$ which will be the case
considered in this work. Therefore, when quoting the
quasi-Einstein metrics we will be referring to the gradient case. However, we shall here obtain some examples of non gradient quasi-Einstein metrics.

From now on, we consider $M^3$ a simply connected homogeneous manifold. Based in the eight model geometries for 3-dimensional manifolds identified by Thurston  (cf. Chapter 3, section 8 in \cite{thurston}, see also Theorem 5.1 in \cite{peter}) we may classify $M^3$ according to its isometry group $Iso(M^{3}),$ whose dimension can be $3$,
$4$ or $6$, detaching that $6$-dimensional are space forms $\Bbb{R}^3,$ $\Bbb{H}^3$ and $\Bbb{S}^3,$ which are Einstein. But  Einstein structures are well known in dimension
$3,$ see e.g. \cite{besse}. So, it remains to describe quasi-Einstein metrics on simply connected homogeneous spaces
with isometry group of dimension $3$ and $4$.

When its isometry group has dimension $3$ the manifold possesses a geometric structure modeled on Lie group $Sol^{3},$ for more details see Theorem 5.2 of \cite{peter}. More exactly, the space $Sol^{3}$ can be viewed as  $\Bbb{R}^{3}$ endowed with the metric
\begin{equation*}
\label{metricsol3}g_{Sol^{3}} = e^{2t}dx^2 + e^{-2t}dy^2 + dt^2,
\end{equation*}where $(x,y,t)$ are canonical coordinates on $\Bbb{R}^{3}.$ It is important to observe that $Sol^{3}$ has a Lie group structure with respect to which the above metric is left-invariant. For more details about the geometry of  $Sol^{3}$ we recommend Section 2 in \cite{daniel}. Concerning to this manifold we have the next result.

\begin{theorem}
\label{thmsol3} $Sol^{3}$ does not carry any quasi-Einstein structure.
\end{theorem}

Proceeding, we treat of 3-dimensional homogeneous manifolds $M^3$ with isometry group of dimension 4. In this case, such a manifold is a Riemannian  fibration onto a $2$-dimensional space form $\mathbb{N}^{2}_{\kappa}$ with constant sectional  curvature $\kappa$. In other words, there is a Riemannian
submersion $\pi\colon M^3 \to \mathbb{N}^{2}_{\kappa}$ with fibers diffeomorphic either to $\mathbb{S}^1$  or to  $\mathbb{R}$, depending whether $M^3$ is compact or not. One remarkable propriety of the vector field $E_3$ tangent to the fibers is that it is a Killing vector field  for which  $\nabla_XE_3 = \tau X \times E_3$ for all $X\in\mathfrak{X}(M)$, where $\tau$ is a constant, called curvature of the bundle, while $\times$ means cross product.

Based on Thruston's classification \cite{thurston} (see also Section 5 in  \cite{peter}), it is well-known that if $M^3$ is non compact it has isometry group of one of the following Riemannian manifolds:
\begin{eqnarray*}
\left \{ \begin{array}{lll}
                     \mathbb{S}^2_{\kappa} \times \mathbb{R},\,\hbox{when}\, \ \kappa > 0,\, \tau = 0,  \\
                    \mathbb{H}^2_{\kappa} \times \mathbb{R},\,\hbox{when}\, \ \kappa < 0 ,\, \tau = 0, \\
                     Nil_{3}(\kappa, \tau),\,\hbox{when}\,  \kappa = 0 ,\, \tau \neq 0,  \\
                    \widetilde{PSl_{2}}(\kappa, \tau),\,\hbox{when}\, \ \kappa < 0 ,\, \tau \neq 0,

                  \end{array} \right.
\end{eqnarray*} where each one of these manifolds can be viewed as $\Bbb{R}^3$ endowed with the following Riemannian metric
\begin{eqnarray*}
g=g_{\kappa,\tau} = \left \{ \begin{array}{ll}
                dx^2+dy^2+[\tau(xdy-ydx)+dt]^2 , \,  \kappa = 0\\
                \rho^2(dx^2+dy^2)+\left[2\kappa\tau \rho \left(ydx-
                xdy\right)+dt\right]^2, \,  \kappa \neq 0.
                \end{array} \right.
\end{eqnarray*} We highlight that $Nil_{3}(\kappa, \tau)$ stands for the classical Heisenberg's space in exponential coordinates. Moreover, $\widetilde{PSl_{2}}$ is the universal cover of the Lie group $PSl_{2}$ endowed with the left-invariant Riemannian metric $g_{\kappa,\tau}.$ Otherwise, when $M^3$ are compact  ($\tau\neq 0$ and $\kappa>0$), these manifolds are fibration over a round sphere. In this case, they have isometry group of Berger's sphere, which will be denoted by $\mathbb{S}^{3}_{\kappa,\tau}.$ In other words, Berger's sphere is a standard $3$-dimensional sphere $\mathbb{S}^3$ endowed with the family of metrics
\begin{equation*}
g_{\kappa,\tau}(X,Y)=\frac{4}{\kappa}\Big[\langle X, Y\rangle + \Big(\frac{4\tau^{2}}{\kappa}-1 \Big)\langle X, V\rangle \langle Y, V\rangle\Big],
\end{equation*}
where $\langle\, , \,\rangle$ stands for the round metric on $\Bbb{S}^{3},\,V_{(z,w)}=(iz,iw)$ for each $(z,w)\in \Bbb{S}^{3}$ and $\kappa,\,\tau$ are real numbers with $\kappa>0$ and $\tau\neq0.$  For comprehensive references on such a classification, we indicate to reader \cite{thurston}, \cite{peter} (Section 5) and \cite{daniel2} (Section 2).

On the other hand, we shall show in Lemma \ref{keylemma} of Section \ref{preliminares} that for a $3$-dimensional homogeneous Riemannian manifold whose group of isometries has dimension $4$ its Ricci tensor satisfies
\begin{equation}
\label{ricc1}
Ric -(4\tau^2-\kappa)E_{3}^{\flat}\otimes E_{3}^{\flat}= (\kappa-2\tau^2)g.
\end{equation}Now taking into account that $E_{3}$ is a Killing vector field, if the vector field  $X = \sqrt{m(4\tau^2-\kappa)}\,E_3$ is well defined, then we have  $\frac{1}{2}\mathcal{L}_{X}g=0$. Whence, letting $\lambda=\kappa-2\tau^2$  we obtain
\begin{equation*}
Ric+\frac{1}{2}\mathcal{L}_{X}g -\frac{1}{m}X^{\flat}\otimes
X^{\flat}= \lambda g,
\end{equation*}which gives the next example:

\begin{example}
\label{ex1qem}Let $(M^3, g_{\kappa,\tau})$  be a $3$-dimensional homogeneous Riemannian manifolds with $4$-dimensional isometry
group such that $X = \sqrt{m(4\tau^2-\kappa)}\,E_3$ is well defined. Letting $\lambda=\kappa-2\tau^2$  we deduce that
$(M^3, g_{\kappa,\tau}, X,\lambda)$ is a $m$-quasi-Einstein metric. We notice that in this case $X$ is not necessarily gradient type.
\end{example}

We point out that if $\lambda>0$ we can prove by a similar argument used in \cite{fr} and \cite{wylie} that
$\big(M^3,g\big)$ is compact. On the other hand, for $\mathbb{S}^2_{\kappa} \times \mathbb{R}$ we have
$4\tau^{2}<\kappa,$ therefore, in the previous example we must have $\big(M^3, g\big)\neq\mathbb{S}_{\kappa}^{2}
\times \mathbb{R}.$  Moreover, we shall show that $\Bbb{H}_{\kappa}^{2}\times \Bbb{R}$ is the unique case for which the
associated vector field is gradient, more precisely, $X$ is the gradient of $f$ given according to Examples \ref{ex1show} and  \ref{ex2show}. In the others cases the associated vector fields are non gradient. Whence, we present  the first
examples of compact and non compact $m$-quasi-Einstein metrics with non gradient vector field making sense the general definition
(\ref{eqprincqem}) of \cite{brG}.

Concerning to Berger's sphere we detach that they admit \emph{shrinking}, \emph{expanding} and \emph{steady} non gradient
$m$-quasi-Einstein metrics, since $\lambda = \kappa-2\tau^{2}$ can assume any sign.

Proceeding it is important to detach that on $\mathbb{H}^2_{\kappa} \times \mathbb{R}$ we have two examples of gradient quasi-Einstein structure. First, we have the following example for a Killing vector field.
\begin{example}
\label{ex1show}
 We consider $\mathbb{H}^2_{\kappa} \times \mathbb{R}$ with its
standard metric and the potential function
$f(x,y,t)=\pm\sqrt{-m\kappa}t+c,$ where $c$ is a constant. It is easy to see that $\nabla
f=\pm\sqrt{-m\kappa}\partial_{t}$, hence $Hess\,f=0.$ Therefore
$(\mathbb{H}_{\kappa}^{2}\times\mathbb{R},\,\nabla f,\,\kappa))$ is a
quasi-Einstein metric.
\end{example}

Next we shall describe our second example to $\mathbb{H}^2_{\kappa} \times \mathbb{R}$, where its
associated vector field is not Killing vector field.
\begin{example}
\label{ex2show} We consider $\mathbb{H}^2_{\kappa} \times \mathbb{R}$ with its
standard metric and the potential function
$f(x,y,t)=-m\ln\cosh\Big[\sqrt{-\frac{\kappa}{m}}(t+a)\Big]+b,$ where $a$ and $b$ are constants. Under these conditions
$(\mathbb{H}^2_{\kappa} \times \mathbb{R},\,\nabla f,\, \kappa)$ is a
quasi-Einstein metric.
\end{example}

Now, it is natural to ask what are the $m$-quasi-Einstein metrics on $M^3$? In fact, those presented in Examples
\ref{ex1show} and  \ref{ex2show} are unique for gradient non compact case. Therefore, we deduce the following uniqueness theorem.
\begin{theorem}
\label{thm2qem} Let $(M^3,\,g,\,\nabla f,\,\lambda)$ be a $3$-dimensional homogeneous quasi-Einstein structure. Then this structure is either Einstein or it corresponds to one of the previous structures given on $\mathbb{H}^2_{\kappa} \times\mathbb{R}$. In particular, $g$ is a static metric provided $m=1.$
\end{theorem}

As a consequence of Theorem \ref{thm2qem} we shall derive the following corollary.

\begin{corollary}
\label{corS2R} $\mathbb{S}^2_{\kappa} \times\mathbb{R},$ $Nil_{3}(\kappa, \tau)$  and $\widetilde{PSl_{2}}(\kappa, \tau)$ do not carry a quasi-Einstein structure.
\end{corollary}

\section{Preliminaries}
\label{preliminares}
 In this section we shall develop a few tools concerning to $3$-dimensional homogeneous manifolds according to the dimension of their isometry group in order to prove our results. For comprehensive references on such a theory, we indicate for instance Thurston's book \cite{thurston}, \cite{peter} and \cite{daniel2}.

 \subsection{$3$-dimensional homogeneous manifold with isometry group of dimension $3$}
As it was previously mentioned $3$-dimensional homogeneous manifold with isometry group of dimension $3$ possesses a geometric structure modeled on the Lie group $Sol^{3}.$ In particular, we may consider $Sol^{3}$ as  $\Bbb{R}^{3}$ endowed with the metric
\begin{equation}
\label{metricsol3}g_{Sol^{3}} = e^{2t}dx^2 + e^{-2t}dy^2 + dt^2.
\end{equation}
Whence, we can check directly that the next set gives  an orthonormal frame on ${Sol}^3$.
\begin{equation}
\label{orthsol3} \{E_1 = e^{-t}\partial_x,\, E_2 =
e^{t}\partial_y,\, E_3 = \partial_t\}.
\end{equation}
By using this frame we obtain the next lemma.
\begin{lemma}
\label{lemconsol3} Let us consider on $Sol^{3}$ the metric and the frame  given, respectively, by (\ref{metricsol3}) and  (\ref{orthsol3}). Then its Riemannian connection $\nabla$ obeys the rules:

\begin{equation}\label{consol3}
                \Bigg \{\begin{array}{lll}
                \nabla_{E_1}E_1 = -E_{3}&  \nabla_{E_1}E_2 =0&  \nabla_{E_1}E_3 = E_{1}\\
                \nabla_{E_2}E_1 = 0&  \nabla_{E_2}E_2 = E_{3} &
                 \nabla_{E_2}E_3 = -E_{2} \\
                \nabla_{E_3}E_1 = 0 & \nabla_{E_3}E_2 = 0 &
                 \nabla_{E_3}E_3 = 0.
              \end{array}  \\
\end{equation}
Moreover, the Lie brackets satisfy:
\begin{equation}
\label{liebrsol}[E_{1},E_{2}]=0, [E_{1},E_{3}]=E_{1} \,
\text{and}\,  [E_{2},E_{3}]=-E_{2}.
\end{equation}
\end{lemma}
Next we use this lemma in order to compute the Ricci tensor of $Sol^3.$ More exactly, we have.

\begin{lemma}
\label{ricsol3}
The Ricci tensor of $Sol^3$ is given by $Ric=-2 E_{3}^{\flat}\otimes
E_{3}^{\flat}.$
\end{lemma}
\begin{proof}Computing $Ric (E_1, E_1)$ with the aid of (\ref{consol3}) and (\ref{liebrsol}) we obtain
\begin{eqnarray*}
Ric (E_1, E_1) &=& \langle\nabla_{E_{2}} \nabla_{E_{1}}E_{1}-\nabla_{E_{1}} \nabla_{E_{2}}E_{1}+\nabla_{[E_{1},E_{2}]} E_{1},E_{2}\rangle\\& + & \langle\nabla_{E_{3}} \nabla_{E_{1}}E_{1}-\nabla_{E_{1}} \nabla_{E_{3}}E_{1}+\nabla_{[E_{1},E_{3}]} E_{1},E_{3}\rangle \\
 &=& \langle-\nabla_{E_{2}} E_{3},E_{2}\rangle+\langle\nabla_{E_{1}} E_{1},E_{3}\rangle =0.\\
\end{eqnarray*}
In a similar way we show that  $Ric (E_i, E_j)=0$, for $i\ne j$ as well as $i=j=2.$ Finally, we have
\begin{eqnarray*}
Ric (E_3, E_3) &=&\langle\nabla_{E_{1}} \nabla_{E_{3}}E_{3}-\nabla_{E_{3}} \nabla_{E_{1}}E_{3}+\nabla_{[E_{3},E_{1}]} E_{3},E_{1}\rangle  \\& + & \langle\nabla_{E_{2}} \nabla_{E_{3}}E_{3}-\nabla_{E_{3}} \nabla_{E_{2}}E_{3}+\nabla_{[E_{3},E_{2}]} E_{3},E_{2}\rangle\\
 &=&\langle\nabla_{[E_{3},E_{1}]} E_{3},E_{1}\rangle  + \langle\nabla_{[E_{3},E_{2}]} E_{3},E_{2}\rangle=-2,\\
\end{eqnarray*}which completes the proof of the lemma.
\end{proof}

Proceeding we have the following lemma for $Sol^{3}$.
\begin{lemma}
\label{lemsol3a}
Suppose that  $\big(Sol^3,\,g,\,X,\,\lambda\big)$ carries  a $m$-quasi-Einstein metric. Then the following statements hold:
\begin{enumerate}
\item $E_{1}\langle X,E_{1}\rangle+\langle X,E_{3}\rangle=\frac{1}{m}\langle X,E_{1}\rangle^2 +\lambda.$
\item $E_{2}\langle X, E_{2}\rangle-\langle X, E_{3}\rangle=\frac{1}{m}\langle X,E_{2}\rangle^2+\lambda.$
\item $E_{3}\langle X, E_{3}\rangle=\frac{1}{m}\langle X, E_{3}\rangle^2+\lambda+2.$
\item $E_{2}\langle X,E_{1}\rangle+E_{1}\langle X,E_{2}\rangle=\frac{2}{m}\langle X,E_{1}\rangle\langle X,E_{2}\rangle.$
\item $E_{3}\langle X,E_{1}\rangle+E_{1}\langle X,E_{3}\rangle-\langle X, E_{1}\rangle=\frac{2}{m}^\langle X, E_{1}\rangle\langle X, E_{3}\rangle.$
\item $E_{3}\langle X,E_{2}\rangle+E_{2}\langle X, E_{3}\rangle+\langle X,E_{2}\rangle=\frac{2}{m}\langle X,E_{2}\rangle\langle X,E_{3}\rangle.$
\end{enumerate}
\end{lemma}
\begin{proof}Firstly, computing $Ric_{X}^{m}(E_1,E_1)$ in equation (\ref{eqprincqem}), we obtain
$$Ric(E_{1},E_{1})+\langle\nabla_{E_{1}}X,E_{1}\rangle-\frac{1}{m}\langle
X,E_{1}\rangle^{2}=\lambda.$$ We may use Lemma \ref{ricsol3} to deduce $Ric(E_{1},E_{1})=0,$ thus, since $\nabla_{E_1}E_1 = -E_{3},$ we obtain the first assertion.

The other ones are obtained by the same way. Computing $Ric_{X}^{m}(E_2,E_2),$ $Ric_{X}^{m}(E_3,E_3),$ $Ric_{X}^{m}(E_2,E_1),$ $Ric_{X}^{m}(E_3,E_1)$  and $Ric_{X}^{m}(E_3,E_2)$ we arrive at (2), (3), (4), (5) and (6). We left their check for the reader. So, we finishes the proof of the lemma.
\end{proof}

\subsubsection{Proof of Theorem \ref{thmsol3}}
\begin{proof}Let us suppose the existence of a gradient quasi-Einstein structure on  $Sol^{3}.$ From item (3) of Lemma \ref{lemsol3a} we have
\begin{equation}
\label{eq1pthmsol3}
\partial_{t}\langle \nabla f,\partial_{t}\rangle=\frac{1}{m}\langle \nabla f, \partial_{t}\rangle^2+\lambda+2,
\end{equation}
where $f$ is the potential function. Under this condition (\ref{eq1pthmsol3}) is a separable ODE and we may use the differentiability of $\langle \nabla f, \partial_{t}\rangle$ to obtain the solutions $\langle \nabla f, \partial_{t}\rangle=\pm\sqrt{-m(\lambda+2)}$ and $\langle \nabla f, \partial_{t}\rangle=-\sqrt{-m(\lambda+2)}\tanh\big[\sqrt{-\frac{\lambda+2}{m}}(\psi+t)\big],$ where $\psi$ is a function that does not depend of variable $t.$ Therefore, we may say that the potential function is either $f=\varphi\pm\sqrt{-m(\lambda+2)}t$ or $f=\varphi-m\ln\cosh\big[\sqrt{-\frac{\lambda+2}{m}}(\psi+t)\big],$ where $\varphi$ does not depend on $t.$ Now, we divide our proof in two cases.

Our first case is when $f=\varphi\pm\sqrt{-m(\lambda+2)}t.$ In this case, we can admit without loss of generality that $f=\varphi+\sqrt{-m(\lambda+2)}t.$ Thus,  item (5) of Lemma \ref{lemsol3a} gives $E_{1}(f)=0,$ which implies from the first item of Lemma \ref{lemsol3a} that $0\leq\sqrt{-\frac{\lambda+2}{m}}=\lambda\leq-2,$ giving a contradiction. This proves the first part of the proof.

Otherwise, if  $f=\varphi-m\ln\cosh\big[\sqrt{-\frac{\lambda+2}{m}}(\psi+t)\big],$ we may substitute $f$ in equation (5) of Lemma \ref{lemsol3a} to obtain
\begin{equation}
\label{eq2pthmsol3}
\partial_{xt}^2 f=\Big(\frac{1}{m}\partial_{t}f+1\Big)\partial_{x}f.
\end{equation}
On the other hand, we have
\begin{equation}
\partial_{xt}^2 f=(\lambda+2)\partial_{x}\psi\hspace{0,05cm}\mbox{sech}^2\Big[\sqrt{-\frac{\lambda+2}{m}}(\psi+t)\Big],
\end{equation}
\begin{equation}
\partial_{x}f=\partial_{x}\varphi-\sqrt{-m(\lambda+2)}\partial_{x}\psi\hspace{0,01cm}\tanh\Big[\sqrt{-\frac{\lambda+2}{m}}(\psi+t)\Big]
\end{equation}and
\begin{equation}
\partial_{t}f=-\sqrt{-m(\lambda+2)}\hspace{0,01cm}\tanh\Big[\sqrt{-\frac{\lambda+2}{m}}(\psi+t)\Big].
\end{equation}
Now, we substitute the last three equations in (\ref{eq2pthmsol3}) to arrive at

\begin{equation}
\sqrt{m}\big(\partial_{x}\varphi-(\lambda+2)\partial_{x}\psi\big)=\sqrt{-(\lambda+2)}\big(\partial_{x}\varphi+m\partial_{x}\psi\big)\tanh\big[\sqrt{-\frac{\lambda+2}{m}}(\psi+t)\big].
\end{equation}
From what it follows that
\begin{equation}
\label{1234}
\partial_{x}\varphi-(\lambda+2)\partial_{x}\psi=0
\end{equation}and

\begin{equation}
\label{123}
\partial_{x}\varphi+m\partial_{x}\psi=0.
\end{equation}

In a similar way we may substitute $f$ in item (6) of Lemma \ref{lemsol3a} to obtain
\begin{equation}
\partial_{yt}^2 f=\Big(\frac{1}{m}\partial_{t}f-1\Big)\partial_{y}f.
\end{equation}
On the other hand, it easy to see that
\begin{equation}
\label{e1psol}
\partial_{yt}^2f = (\lambda+2)\partial_y\psi\hspace{0,05cm}\mbox{sech}^2\hspace{-0,1cm}\left[\sqrt{-\frac{\lambda+2}{m}}(\psi+t)\right],
\end{equation}

\begin{equation}
\label{e2psol}
\partial_{y}f = \partial_y\varphi-\sqrt{-m(\lambda+2)}\partial_y\psi\hspace{0,01cm}\tanh\left[\sqrt{-\frac{\lambda+2}{m}}(\psi+t)\right]
\end{equation}
and
\begin{equation}
\label{e3psol}
\partial_{t}f = -\sqrt{-m(\lambda+2)}\tanh\left[\sqrt{-\frac{\lambda+2}{m}}(\psi+t)\right].
\end{equation}
Therefore, substituting (\ref{e1psol}), (\ref{e2psol}) and (\ref{e3psol}) in item (4) of Lemma \ref{lemsol3a} we have
$$\sqrt{m}\hspace{0,05cm}[\partial_y\varphi+(\lambda+2)\partial_y\psi] = - \sqrt{-(\lambda+2)}[\partial_y\varphi-m\partial_y\psi]\tanh\left[\sqrt{-\frac{\lambda+2}{m}}(\psi+t)\right],$$ which implies
\begin{eqnarray}\label{T2.1/5}
\partial_y\varphi+(\lambda+2)\partial_y\psi = 0
\end{eqnarray}
and
\begin{eqnarray}\label{T2.1/6}
\partial_y\varphi-m\partial_y\psi = 0.
\end{eqnarray}

Now, we consider $\lambda\neq-m-2,$ therefore, from (\ref{1234}), (\ref{123}), (\ref{T2.1/5}) and (\ref{T2.1/6}) we have $$\partial_x\varphi = \partial_y\varphi = \partial_x\psi = \partial_y\psi = 0,$$ which implies that $f$ does not depend on $x$ and $y.$ Moreover, by using items (1) and (2) of Lemma \ref{lemsol3a} we conclude that  $f$ does not depend on $t$ and then $f$ is constant, which is a contradiction.

Since $\lambda=-m-2,$ we may derive item (1) of Lemma \ref{lemsol3a} with respect to $t,$ item (5) with respect to $x$ and we compare the results to arrive at
\begin{equation}
\label{eq1sol123}
(\partial_tf+m)\partial_{xx}^2f - \partial_xf\partial_{xt}^2f = -m[2(\partial_tf-\lambda)+\partial_{tt}^2f]e^{2t}.
\end{equation}
Next, we substitute item (5) of Lemma \ref{lemsol3a} in (\ref{eq1sol123}) to obtain
$$(\partial_tf+m)\left[\partial_{xx}^2f - \frac{1}{m}(\partial_xf)^2\right] = -m[2(\partial_tf-\lambda)+\partial_{tt}^2f]e^{2t},$$
thus we may use again item (1) of Lemma \ref{lemsol3a} to conclude $$(\partial_tf-m)(\partial_tf-\lambda) = m\partial_{tt}^2f.$$ Finally, it suffices to use item (3) of Lemma \ref{lemsol3a} to obtain $\partial_{t}f=m,$ which is a contradiction. So, we finishe the proof of the theorem.

\end{proof}

\subsection{Non compact $3$-dimensional homogeneous manifold with isometry group of dimension $4$}
We recall that the projection $\pi : M^3 \rightarrow N_{\kappa}^{2}$, given by $\pi(x, y, t) = (x, y)$ is a submersion, where  $N_{\kappa}^{2}$ is endowed with its canonical metric $ds^{2}=\rho^{2}(dx^{2}+dy^{2}),$ while $\rho=1$ or $\rho=\frac{2}{1+\kappa(x^{2}+y^{2})}$ according to $\kappa=0$ or $\kappa\neq0$, respectively. The natural orthonormal frame on $N_{\kappa}^{2}$ is given by $\{e_{1}=\rho^{-1}\partial_{x}, e_{2}=\rho^{-1}\partial_{y}\}.$ Moreover, translations along  the fibers are isometries, therefore $E_3$ is a Killing vector field. Thus, considering horizontal lifting of  $\{e_1, e_2\}$ we obtain $\{E_1, E_2\}$, which jointly with $E_3$ gives an orthonormal frame $\{E_1, E_2, E_3\}$ on $M^3.$ In addition,  since $\{\partial _x, \partial _y\}$ is a natural frame for $N_\kappa^2,$ then a natural frame  for $M^3$ is $\{\partial _x, \partial _y, \partial _t\},$ where $\partial _t$ is tangent to the fibers. Using this frame we have the following lemma for a non compact $3$-dimensional homogeneous manifold which can be found in \cite{thurston}.

\begin{lemma}
\label{lem1} Rewriting the referential $\{E_1, E_2, E_3\}$ in terms of
$\{\partial _x, \partial _y, \partial _t\}$, we have:
\begin{enumerate}
\item If $\kappa   \neq 0,$ then
$E_{1}=\frac{1}{\rho}\partial_{x}+2\kappa\tau y
\partial_{t},$
$E_{2}=\frac{1}{\rho}\partial_{y}-2\kappa\tau x\partial_{t}$
and $E_{3}=\partial_{t}.$

\item If $\kappa=0,$ then $E_{1}=\partial_{x}-\tau y\partial_{t},$ $E_{2}=\partial_{y}+\tau x\partial_{t}$ and $E_{3}=\partial_{t}.$
\end{enumerate}
Moreover, endowing $M^3$ with the metric
\begin{eqnarray*}
g = \left \{ \begin{array}{ll}
                dx^2+dy^2+[\tau(xdy-ydx)+dt]^2 , \,  \kappa = 0\\
                \rho^2(dx^2+dy^2)+\left[2\kappa\tau \rho \left(ydx-
                xdy\right)+dt\right]^2, \,  \kappa \neq 0,
                \end{array} \right.
\end{eqnarray*}we have the following identities for its Riemannian connection $\nabla$:

\begin{equation} \left\{\begin{array}{lll}
                \nabla_{E_1}E_1 =  \kappa yE_2 &  \nabla_{E_1}E_2 =
                -\kappa yE_1+\tau E_3 &  \nabla_{E_1}E_3 = - \tau E_2\\
                \nabla_{E_2}E_1 = -\kappa xE_2-\tau E_3 &  \nabla_{E_2}E_2 = \kappa xE_1 &
                 \nabla_{E_2}E_3 = \tau E_1 \\
                \nabla_{E_3}E_1 = - \tau E_2 &  \nabla_{E_3}E_2 = \tau E_1 &
                \nabla_{E_3}E_3 = 0.
                \end{array}\right.  \\
\end{equation}
                In particular, we obtain from the above identities the following relations for the Lie brackets:
\begin{equation}\label{bracket12}
[E_{1},E_{2}]=-\kappa yE_1+\kappa xE_2+2\tau E_3
\end{equation}
and
\begin{equation}\label{bracketi3}
 [E_{1},E_{3}]=[E_{2},E_{3}]=0.\\
\end{equation} Moreover, up to isometries, we may assume that $\kappa=-1,0$ or $1.$
\end{lemma}

\subsection{Compact $3$-dimensional homogeneous manifold with isometry group of dimension $4$}
They occur when $\tau\neq 0$ and $\kappa>0,$ more precisely, we recall that compact homogeneous Riemannian manifolds are Berger's
spheres; see \cite{peter} and \cite{thurston}.  For sake of completeness and to keep the same notation we shall choose the next construction for  Berger's spheres, for more details see
\cite{besse}. In what follow,  Berger's sphere is a standard $3$-dimensional sphere
\begin{eqnarray*}
\mathbb{S}^3 =\{(z,w)\in \Bbb{C}^{2};\,|z|^{2}+|w|^{2}=1\}
\end{eqnarray*}
endowed with the family of metrics
\begin{equation*}
g_{\kappa,\tau}(X,Y)=\frac{4}{\kappa}\Big[\langle X, Y\rangle + \Big(\frac{4\tau^{2}}{\kappa}-1 \Big)\langle X, V\rangle \langle Y, V\rangle\Big],
\end{equation*}
where $\langle\, , \,\rangle$ stands for the round metric on $\Bbb{S}^{3},\,V_{(z,w)}=(iz,iw)$ for each $(z,w)\in \Bbb{S}^{3}$ and $\kappa,\,\tau$ are real numbers with $\kappa>0$ and $\tau\neq0.$ In particular, $g_{4,1}$ is the round metric. In addition,  Berger's sphere $(\Bbb{S}^{3},g_{\kappa,\tau})$ will be denoted by  $\Bbb{S}^{3}_{\kappa,\tau},$ which is a model for a homogeneous space $M^3$ when $\kappa>0$ and $\tau\neq 0.$ In this case the vertical Killing vector field is given by $E_{3}=\frac{\kappa}{4\tau}V.$ In order to obtain an orthonormal frame we choose $E_{1}(z,w)=\frac{\sqrt{\kappa}}{2}(-\overline{w},\overline{z})$ and $E_{2}(z,w)=\frac{\sqrt{\kappa}}{2}(-i\overline{w},i\overline{z}).$

 Using this frame we have the following lemma for $\Bbb{S}^{3}_{\kappa,\tau}$ which can be found in \cite{besse}.
 \begin{lemma}
 \label{lem1berger} The Riemannian connection $\nabla$ on $\Bbb{S}^{3}_{\kappa,\tau}$ is determined by
\begin{equation} \left\{\begin{array}{lll}
                \nabla_{E_1}E_1 = 0 &  \nabla_{E_1}E_2 = - \tau E_3 &  \nabla_{E_1}E_3 =  \tau E_2\\
                \nabla_{E_2}E_1 =  \tau E_3&  \nabla_{E_2}E_2 = 0 &
                 \nabla_{E_2}E_3 = - \tau E_1 \\
                \nabla_{E_3}E_1 = -\displaystyle\frac{\kappa-2\tau^{2}}{2\tau }E_2 & \nabla_{E_3}E_2 = \displaystyle
                \frac{\kappa-2\tau^{2}}{2 \tau }E_1 &
                 \nabla_{E_3}E_3 = 0.
              \end{array}\right.  \\
\end{equation}It is immediate to verify that the Lie brackets satisfy:
 \begin{equation}
 \label{bracketberger}
 [E_{1},E_{2}]=-2\tau E_{3},\, [E_{2},E_{3}]=-\frac{\kappa}{2 \tau} E_{1},\,[E_{1},E_{3}]=\frac{\kappa}{2\tau }E_{2}.
 \end{equation}

 \end{lemma}

\section{Key Results}

As a consequence of Lemmas \ref{lem1} and \ref{lem1berger} we can explicit the Ricci tensor of a $3$-dimensional homogeneous Riemannian manifold with $4$-dimensional isometry group according to next lemma.

\begin{lemma}
\label{keylemma}Let $(M^3,\,g)$ be a $3$-dimensional
homogeneous Riemannian manifold with $4$-dimensional isometry
group. Then, each frame $\{E_1, E_2, E_3\}$ constructed before on $M^3$ diagonalizes the Ricci tensor. More precisely, we have
\begin{equation}
\label{qem1}
Ric = (\kappa-2\tau^2)g - (\kappa-4\tau^2)E_{3}^{\flat}\otimes
E_{3}^{\flat}.
\end{equation}
\end{lemma}

\begin{proof} Firstly, we consider $(M^3,\,g)$ a non compact $3$-dimensional
homogeneous Riemannian manifold with $4$-dimensional isometry group. Since we  can write the Ricci tensor as follows
\begin{eqnarray}\label{P2.1/1}
Ric (X, Y) =  \sum_{j, k =1}^{3} \langle X, E_j \rangle \langle Y,
E_k \rangle Ric (E_j, E_k),
\end{eqnarray}in order to find $Ric (E_j, E_k)$  we shall show that $Ric (E_j, E_k )=\lambda_{j}\delta_{jk}.$ Indeed, using Lemma \ref{lem1} we have

\begin{eqnarray*}
Ric (E_1, E_1) &=& \langle\nabla_{E_{2}} \nabla_{E_{1}}E_{1}-\nabla_{E_{1}} \nabla_{E_{2}}E_{1}+\nabla_{[E_{1},E_{2}]} E_{1},E_{2}\rangle\\& + & \langle\nabla_{E_{3}} \nabla_{E_{1}}E_{1}-\nabla_{E_{1}} \nabla_{E_{3}}E_{1}+\nabla_{[E_{1},E_{3}]} E_{1},E_{3}\rangle \\
&=& \langle\nabla_{E_{2}} \big(\kappa y E_{2}\big)+\nabla_{E_{1}} \big(\kappa x E_{2}+\tau E_{3}\big)-\kappa y\nabla_{E_{1}}E_{1}+\kappa x\nabla_{E_{2}}E_{1}+2\tau \nabla_{E_{3}}E_{1},E_{2}\rangle\\& + & \langle\nabla_{E_{3}} \big(\kappa yE_{2}\big)+\nabla_{E_{1}} \big(\tau E_{2}\big),E_{3}\rangle \\
&=& \frac{2\kappa}{\rho} - \kappa^2(x^2+y^2) - 2\tau^2 =
\kappa-2\tau^2.
\end{eqnarray*}
In a similar way we have $Ric (E_2, E_2)= \kappa-2\tau^2$ and $Ric (E_3, E_3)= 2\tau^2.$

Now we claim that  $Ric (E_j, E_k) = 0$ for $j \neq k$. In fact, let us compute only $Ric(E_1, E_2),$ since the others terms follow mutatis mutandis.
\begin{eqnarray*}
Ric (E_1, E_2) &=& \langle\nabla_{E_{3}} \nabla_{E_{1}}E_{2}-\nabla_{E_{1}} \nabla_{E_{3}}E_{2}+\nabla_{[E_{1},E_{3}]} E_{2},E_{3}\rangle\\
&=& \langle\nabla_{E_{3}} \big( \kappa y E_{1}+\tau E_{3}\big)-\nabla_{E_{1}} \big(\tau E_1\big),E_{3}\rangle\\
&=& \kappa y\langle\nabla_{E_{3}} E_{1},E_{3}\rangle-\tau \langle \nabla_{E_{1}} E_1,E_{3}\rangle=0,
\end{eqnarray*}which finishes our claim. Therefore, using $(\ref{P2.1/1})$, we deduce
\begin{eqnarray}
Ric = (\kappa-2\tau^2)g - (\kappa-4\tau^2)E_{3}^{\flat}\otimes
E_{3}^{\flat},
\end{eqnarray}which completes the proof in this case. We now point out that using Lemma \ref{lem1berger},  straightforward computations as above give the same result for Berger's spheres. Then we complete the proof of the lemma.
\end{proof}

\begin{lemma}
\label{lemkey2}Let $\big(M^3,\,g,\,X,\,\lambda
\big)$ be a non compact $3$-dimensional homogeneous
quasi-Einstein metric with $4$-dimensional isometry group. If
$E_1, E_2$ and $E_3$ are given by Lemma \ref{lem1}, then hold:
\begin{equation}
\label{e1.lem2} E_1\langle X, E_1 \rangle - \kappa y\langle X, E_2
\rangle =\frac{1}{m}\langle X, E_1\rangle^{2}+\lambda -
(\kappa-2\tau^2).
\end{equation}
\begin{equation} \label{e2.lem2}
E_2\langle X, E_2 \rangle - \kappa x\langle X, E_1 \rangle =
\frac{1}{m}\langle X, E_2\rangle^{2}+\lambda - (\kappa-2\tau^2).
\end{equation}
\begin{equation}\label{e3.lem2}
E_3\langle X, E_3 \rangle  = \frac{1}{m}\langle X, E_3\rangle^{2}
+\lambda - 2\tau^2.
\end{equation}
\begin{equation}\label{e4.lem2}
E_2\langle X, E_1 \rangle + E_1\langle X, E_2 \rangle + \kappa(y
\langle X, E_1 \rangle + x \langle X, E_2 \rangle) =
\frac{2}{m}\langle X, E_1\rangle \langle X, E_2\rangle.
\end{equation}

\begin{equation}\label{e5.lem2}
E_3\langle X, E_1 \rangle + E_1\langle X, E_3 \rangle + 2\tau
\langle X, E_2 \rangle = \frac{2}{m}\langle X, E_1\rangle \langle
X, E_3\rangle.
\end{equation}

\begin{equation}\label{e6.lem2}
E_3\langle X, E_2 \rangle + E_2\langle X, E_3 \rangle - 2\tau
\langle X, E_1 \rangle = \frac{2}{m}\langle X, E_2\rangle \langle
X, E_3\rangle.
\end{equation}

\end{lemma}

\begin{proof}We notice that by using equation (\ref{eqprincqem}) we can write
\begin{equation}\label{L3.2/1}
\mathcal{L}_Xg(E_i, E_j) = 2 \left(\lambda \delta_{ij} - Ric
(E_i,E_j)+\frac{1}{m}\langle X,E_i\rangle \langle X,E_j\rangle
\right).
\end{equation}
Taking into account that $\mathcal{L}_Xg(E_i, E_j)=\langle \nabla_{E_i}X,E_j\rangle+\langle \nabla_{E_j}X,E_i\rangle$ we use the compatibility of
the metric $g$ to infer
\begin{equation}\label{L3.2/2}
E_i \langle X, E_j \rangle + E_j \langle X, E_i \rangle - \langle
X, \nabla_{E_i}E_j + \nabla_{E_j}E_i \rangle = 2 \left(\lambda
\delta_{ij} - R_{ij}+\frac{1}{m}X_{i}X_{j}\right).
\end{equation}
Therefore, using Lemma \ref{lem1} and (\ref{L3.2/2}) straightforward computations give the desired statements.
\end{proof}

\section{Proof of Theorem \ref{thm2qem}}

\begin{proof}
Since $M^{3}$ is a $3$-dimensional homogeneous manifold, its
isometry group has dimension $3$, $4$ or $6$. Making use of
Theorem \ref{thmsol3} we can discard $Sol^{3}$. Moreover, when
$dim\,Iso\big(M^3,g\big)=6$ we have  space forms,
which give Einstein metrics. Therefore, it remains to describe
gradient quasi-Einstein structures of homogeneous spaces with isometry
group of dimension $4.$

To this end, we start solving  ODE  (\ref{e3.lem2}) to conclude that either $$\langle \nabla f,E_{3}\rangle=\pm\sqrt{-m(\lambda-2\tau^2)}$$ or $$\langle \nabla f, E_3 \rangle = - \sqrt{-m(\lambda-2\tau^2)}\tanh\left[\sqrt{-\frac{\lambda-2\tau^2}{m}}(\psi+t)\right],$$ where $\psi\in C^{\infty}(M^3)$ does not depend on $t.$ Taking into account this two possibilities to $\langle\nabla f, E_{3}\rangle$ we conclude that $f$ is given by, either
\begin{eqnarray*}
f = \varphi \pm \sqrt{-m(\lambda-2\tau^2)}\hspace{0,05cm}t,
\end{eqnarray*}
or
\begin{eqnarray*}
f = \varphi - m \log \cosh \left[\sqrt{-\frac{\lambda-2\tau^2}{m}}(\psi+t)\right],
\end{eqnarray*}
where $\varphi\in C^{\infty}(M^3)$ does not depend on $t.$

Now, we shall divide this part of the proof in two cases.

First, if $f = \varphi \pm \sqrt{-m(\lambda-2\tau^2)}\hspace{0,05cm}t,$ then we substitute $f$ in equations (\ref{e5.lem2}) and (\ref{e6.lem2}), respectively,  to arrive at

\begin{eqnarray}
\label{eq1pthmA}
\tau E_2(f) = \pm \sqrt{-\frac{\lambda-2\tau^2}{m}}E_1(f)
\end{eqnarray}
and
\begin{eqnarray}
\label{eq2pthmA}
\tau E_1(f) = \mp \sqrt{-\frac{\lambda-2\tau^2}{m}}E_2(f).
\end{eqnarray}
On the other hand, by using equation (3.12) of Lemma 3.2 and item $(b)$ of Proposition 3.6, both in \cite{csw}, we conclude that $\lambda$ and $\tau$ can not be zero simultaneously.

Therefore, our two possibilities (\ref{eq1pthmA}) give
\begin{equation}
\label{eq3plll}
E_1(f) = E_2(f) = 0,
\end{equation} which substituted in (\ref{e1.lem2}) implies $\lambda=\kappa-2\tau^2.$ We notice that $\mathbb{S}_{\kappa}^{2}
\times \mathbb{R}$ has $\tau=0,$ therefore, if $\big(M^3, g\big)=\mathbb{S}_{\kappa}^{2}
\times \mathbb{R}$ we can use Qian's theorem \cite{qian} to conclude $\mathbb{S}_{\kappa}^{2}
\times \mathbb{R}$ is compact, which is a contradiction, see also \cite{kk} and \cite{fr}.  On the other hand, since $[E_{1},E_{2}]=-\kappa yE_{1}+\kappa x E_{2}+2\tau E_{3},$ we can use  (\ref{eq3plll}) to obtain $2\tau E_{3}(f)=0.$ From what it follows that $\big(M^3, g\big)$ can not be $Nil_3(\kappa, \tau)$ and $\widetilde{PSl_2}(\kappa, \tau).$ Therefore $M^3 =
\Bbb{H}_{\kappa}^2 \times \Bbb{R}$ and $\lambda = \kappa$, which finishes the first case.

Proceeding we consider $f = \varphi - m \log \cosh \left[\sqrt{-\frac{\lambda-2\tau^2}{m}}(\psi+t)\right]$. In this case we start supposing that $M^3= Nil_3(\kappa, \tau).$ Therefore, from (\ref{e5.lem2}) we obtain

\begin{eqnarray}\label{T4.1/1}
E_1E_3(f)+\tau E_2(f) = \frac{1}{m}E_1(f)E_3(f).
\end{eqnarray}
From what it follows that
\begin{eqnarray*}
\partial_{xt}^2f-\tau yE_3E_3(f) + \tau E_2(f) = \frac{1}{m}[\partial_{x}f-\tau yE_3(f)]E_3(f),
\end{eqnarray*}
which compared with (\ref{e3.lem2}) gives
\begin{eqnarray}
\label{eqzxc}
\partial_{xt}^2f-\tau (\lambda - 2\tau^2)y + \tau E_2(f) = \frac{1}{m}\partial_{x}f\partial_t f.
\end{eqnarray}
Substituting the value of $f$ in (\ref{eqzxc}) we obtain
\begin{eqnarray*}
& & \sqrt{-(\lambda-2\tau^2)}[m \tau(\partial_y\psi+\tau x)-\partial_x\varphi]\tanh\left[\sqrt{-\frac{(\lambda-2\tau^2)}{m}}(\psi+t)\right]\\
&=& \sqrt{m}[(\lambda - 2\tau^2)(\partial_x\psi-\tau y)+\tau \partial_y\varphi].
\end{eqnarray*}
Now, we notice that the right side of the previous expression does not depend on $t,$ thus, since $\lambda-2\tau^2\neq 0$ we have $\tanh\left[\sqrt{-\frac{(\lambda-2\tau^2)}{m}}(\psi+t)\right]\neq 0,$ which implies that
\begin{eqnarray}\label{T4.1/2}
\partial_x\varphi - m \tau \partial_y\psi = m \tau^2 x
\end{eqnarray}
and
\begin{eqnarray}\label{T4.1/3}
\tau \partial_y\varphi + (\lambda - 2\tau^2)\partial_x\psi = \tau(\lambda - 2\tau^2)y.
\end{eqnarray}

In a similar way we use equations (\ref{e3.lem2}) and
(\ref{e6.lem2}) to obtain

\begin{eqnarray}\label{T4.1/5}
\partial_y\varphi + m \tau \partial_x\psi = m \tau^2 y
\end{eqnarray}
and
\begin{eqnarray}\label{T4.1/6}
\tau \partial_x\varphi - (\lambda - 2\tau^2)\partial_y\psi =
\tau(\lambda - 2\tau^2)x.
\end{eqnarray}

Now, we may combine (\ref{T4.1/2}) with (\ref{T4.1/6}) and
(\ref{T4.1/3}) with (\ref{T4.1/5}) to obtain, respectively,
$\partial_{y}\psi=-\tau x$ and $\partial_{x}\psi=\tau y,$ which
gives $\tau=0.$ So, we obtain a contradiction.

Therefore, since $M^3 \neq Nil_3(\kappa, \tau),$
we can use (\ref{e5.lem2}) and (\ref{e3.lem2}) to arrive at
\begin{eqnarray}\label{T4.1/7}
\frac{1}{\rho}\partial_{xt}^2f+2\kappa\tau(\lambda - 2\tau^2)y +
\tau E_2(f) = \frac{1}{m\rho}\partial_{x}f\partial_t f.
\end{eqnarray}
Now, we substitute the value of $f$ in (\ref{T4.1/7}) to obtain
\begin{eqnarray*}
& & \sqrt{-(\lambda-2\tau^2)}\left[\frac{1}{\rho}\left(m\tau\partial_y\psi-\partial_x\varphi\right)-2m\kappa\tau^2 x\right]\tanh\left[\sqrt{-\frac{(\lambda-2\tau^2)}{m}}(\psi+t)\right]\\
&=& \sqrt{m}\left\{\frac{1}{\rho}\left[(\lambda -
2\tau^2)\partial_x\psi+\tau
\partial_y\varphi\right]+2\kappa\tau(\lambda - 2\tau^2) y\right\},
\end{eqnarray*}
by using a similar argument used previously we conclude
\begin{eqnarray}\label{T4.1/8}
\partial_x\varphi - m \tau \partial_y\psi = -2m\kappa\tau^2 x\rho
\end{eqnarray}
and
\begin{eqnarray}\label{T4.1/9}
\tau \partial_y\varphi + (\lambda - 2\tau^2)\partial_x\psi =
-2\kappa\tau(\lambda - 2\tau^2)y\rho.
\end{eqnarray}

Analogously, from (\ref{e3.lem2}) and (\ref{e6.lem2}) we have
\begin{eqnarray}\label{T4.5/4}
\frac{1}{\rho}\partial_{yt}^2f-2\kappa\tau (\lambda - 2\tau^2)x - \tau E_1(f) = \frac{1}{m}\partial_{y}f\partial_t
f.
\end{eqnarray}
Substituting the value of $f$ we obtain
\begin{eqnarray*}
&
&\sqrt{-(\lambda-2\tau^2)}\left[\frac{1}{\rho}\left(m\tau\partial_x\psi+\partial_y\varphi\right)+2m\kappa\tau^2
y
\right]\tanh\left[\sqrt{-\frac{(\lambda-2\tau^2)}{m}}(\psi+t)\right]\\
&=& - \sqrt{m}\left\{\frac{1}{\rho}\left[(\lambda -
2\tau^2)\partial_y\psi - \tau
\partial_x\varphi)\right]-2\kappa\tau(\lambda-2\tau^2)\right\},
\end{eqnarray*}
which gives
\begin{eqnarray}\label{T4.1/10}
\partial_y\varphi + m \tau \partial_x\psi = -2m\kappa\tau^2 y\rho
\end{eqnarray}
and
\begin{eqnarray}\label{T4.1/11}
\tau \partial_x\varphi - (\lambda - 2\tau^2)\partial_y\psi =
-2\kappa\tau(\lambda - 2\tau^2)x\rho.
\end{eqnarray}
Finally, we may combine (\ref{T4.1/8}) with (\ref{T4.1/11}) and
(\ref{T4.1/9}) with (\ref{T4.1/10}), respectively, to obtain
$\partial_{y}\psi=2\kappa\tau x\rho$ and $\partial_{x}\psi=-2\kappa\tau y\rho,$ hence $\tau=0.$ Since
$\tau=0$ we can use (\ref{T4.1/8}), (\ref{T4.1/9}), (\ref{T4.1/10})
and (\ref{T4.1/11}) to conclude that $\varphi$ and $\psi$ are
constants. Thus, $f$ depends only on $t$ and then from
(\ref{e6.lem2}) we have $\lambda=\kappa<0$ and $M^3= \Bbb{H}_{\kappa}^2 \times \Bbb{R}$, given according to
Example \ref{ex2show} that was to be proved.
\end{proof}

\begin{acknowledgement}The authors would like to thank the referee for comments and
valuable suggestions. The second author is grateful to J. Case and P. Petersen for helpful conversations about this subject. \nonumber
\end{acknowledgement}

\end{document}